\documentclass[12pt]{article}
\usepackage{fullpage} %Makes text take up more of the page
\usepackage{ifdraft}
\usepackage{amssymb,amsmath,amscd,epsfig,amsthm,verbatim,eqnarray,hyperref}
\usepackage{dsfont,graphicx}
\usepackage{verbatim} %Makes text take up more of the page
\newcommand{\Aut}{\mathop{\rm Aut}\nolimits}
\newcommand{\Aff}{\mathop{\rm Aff}\nolimits}
\newcommand{\SHAut}{\mathop{\rm SHAut}\nolimits}

\newcommand{\G}{\mathcal G}

\newcommand{\A}{\mathcal A}

\newcommand{\Z}{\mathbb Z}
\newcommand{\Sym}{\mathop{\rm Sym}\nolimits}

\newcommand{\bb}{\mathbf b}
\newcommand{\bx}{\mathbf x}
\newcommand{\ba}{\mathbf a}
\newcommand{\bc}{\mathbf c}
\newcommand{\be}{\mathbf e}

\usepackage{xcolor}

\newtheorem{theorem}{Theorem}[section]
\newtheorem{corollary}[theorem]{Corollary}
\newtheorem{proposition}[theorem]{Proposition}
\newtheorem{lemma}[theorem]{Lemma}

\newtheorem{question}{Question}

\theoremstyle{definition}
\newtheorem{definition}{Definition}

\newtheorem*{theoremnormalizer}{Theorem \ref{thm:normalizer}}
\newtheorem*{theoremG}{Theorem \ref{thm:structure_of_G}}

\newtheorem*{theoremL}{Theorem \ref{thm:state-closed_lamplighter}}

\title{Affine Automorphisms of Rooted Trees}
\author{Dmytro M. Savchuk\footnote{Partially Supported by the Proposal Enhancement Grant from University of South Florida and the Simons Collaboration Grant \#317198 from Simons Foundation}\\
               Department of Mathematics and Statistics\\
               University of South Florida\\
               4202 E Fowler Ave\\
               Tampa, FL 33620-5700\\
               \href{mailto:savchuk@usf.edu}{savchuk@usf.edu}\\
               \and
        Said N. Sidki\\
               Departamneto de Matematica,\\
               Universidade de Bras\'{\i}lia,\\
               Brasilia-DF 70910, Brazil\\
               \href{mailto:ssidki@gmail.com}{ssidki@gmail.com}
}

\begin{document}
\maketitle

\begin{abstract}
We introduce a class of automorphisms of rooted $d$-regular trees arising from affine actions on their boundaries viewed as infinite dimensional vector spaces. This class includes, in particular, many examples of self-similar realizations of lamplighter groups. We show that for a regular binary tree this class coincides with the normalizer of the group of all spherically homogeneous automorphisms of this tree: automorphisms whose states coincide at all vertices of each level. We study in detail a nontrivial example of an automaton group that contains an index two subgroup with elements from this class and show that it is isomorphic to the index 2 extension of the rank 2 lamplighter group $\Z_2^2\wr\Z$.
\end{abstract}

%\begin{keyword}
%\end{keyword}

%\linenumbers

\section*{Introduction}

The rooted $d$-regular tree $T_d$ can be naturally identified with the set $X^*$ of all finite words over a finite alphabet $X=\{0,\ldots,d-1\}$. Its boundary, consisting of infinite paths starting from the root without backtracking, can be identified with the set $X^\infty$ of infinite words $a_0a_1\ldots a_i\ldots$ over $X$. Each such infinite word can be represented as an element $a_0 + a_1t+\cdots+a_it^i+\cdots$ of the ring $\Z_d[[t]]$ of formal power series with coefficients in $\Z_d$. Therefore automorphisms from the group $\Aut(X^*)$ of all automorphisms of $T_d$ translate as transformations of $\Z_d[[t]]$.

For example, the automorphism which permutes only the first letter of the input word by the long cycle $\sigma=(0,1,\ldots,d-1)\in\Sym(X)$, translates as a transformation of $\Z_d[[t]]$ defined by $p(t)\mapsto 1+p(t)$. Similarly, the operation of addition by a power series (respectively, by a polynomial) $f(t)=a_0 + a_1t+\cdots+a_it^i+\cdots$ in $\Z_d[[t]]$  corresponds to (respectively, the finitary) automorphism of $T_d$ permuting the $i$-th letter of the input word by $\sigma^{a_{i-1}}$.

Another useful notation that we will use throughout the paper is the following. For an automorphism $g\in\Aut(X^*)$ we denote by $g^{(n)}$ an automorphism of $X^*$ acting trivially on the $n$-th level, and whose states at all vertices of $X^n$ are equal to $g$. For example, $g^{(0)}=g$, $g^{(1)}=(g,g)$, etc. Particularly, we denote by $\sigma^{(n)}$, $n\geq 0$ the automorphism of $X^*$ that acts on the $(n+1)$-st coordinate in the input word by permutation $\sigma$. With this notation the addition of $t^n$ in $\Z_d[[t]]$ exactly corresponds to $\sigma^{(n)}$, and thus the group of automorphisms induced by addition of all possible polynomials in $\Z_d[[t]]$ is the state-closed abelian group $\Delta=\langle \sigma^{(0)},\sigma^{(1)},\ldots,\sigma^{(n)},\ldots\rangle$ consisting of finitary automorphisms of $X^*$.

It was realized in~\cite{gns00:automata}, that the transformations $f(t)\mapsto f(t)+1$ and $f(t)\mapsto (1+t)f(t)$ of $\Z_p[[t]]$ correspond respectively to the automorphisms of the binary tree $a=(0,1)$ and $b=(b,ba)=b^{(1)}(1,\sigma)$ which are the canonical generators of the, so-called, lamplighter group $\Z_2\wr\Z$ studied by Grigorchuk and Zuk~\cite{grigorch_z:lamplighter}. Variants of the lamplighter group had also appeared as normalizers of the group of finitary automorphisms and of its subgroup $\Delta$ (see~\cite{brunner_s:one-rooted97}).

Silva and Steinberg had shown in~\cite{silva_s:lamplighter05} that if $G$ is a finite abelian group then the restricted wreath product $G\wr\Z$ has a faithful representation as an automaton group. Furthermore, Bartholdi and \v Suni\'c in~\cite{bartholdi_s:bsolitar} produced for $G=\Z_d^k$ a different representation of $G\wr\Z$ modeling the multiplication in $\Z_d[[t]]$ by an arbitrary monic polynomial of degree $k$.
Representations of the lamplighter type groups $\Z_d^k\wr\Z$ on the rooted trees arose also in the context of groups acting essentially freely in~\cite{grigorch_s:essfree}, and in a connection to bireversible automata in~\cite{bondarenko_dr:lamplighter}.

All the known representations of lamplighter groups $\Z_d^k\wr\Z=\oplus_{\Z}\Z_d^k\rtimes\Z$ as automaton groups share a common property: the base group $\oplus_{\Z}\Z_d^k$ always consists of commuting automorphisms that act identically at all vertices of each level. More formally, we call an automorphism of the tree $T_d$ \emph{spherically homogeneous} (see~\cite{grigorch_s:essfree}) provided that for each level its states at the vertices of this level all coincide. Thus, each automorphism has a form $a=(b,b,\ldots,b)\sigma_1$, $b=(c,c,\ldots,c)\sigma_2,\ldots$, where $\sigma_i$'s are permutations of $X$. Clearly, all such automorphisms form an uncountable subgroup $\SHAut(X^*)$ of $\Aut(X^*)$ isomorphic to the direct product of countably many copies of $\Sym(X)$. In the case of binary tree this group is abelian, and when $d\geq 3$, it contains an abelian subgroup, which we will denote by $\Aff_{I}(X^*)$, consisting of all spherically homogeneous automorphisms whose permutations at all vertices are powers of the long cycle $\sigma$ (i.e., $\Aff_I(X^*)=\SHAut(X^*)\cap\bigl(\wr_{i=1}^{\infty}\Z_d\bigr)$). This group also can be described as the topological closure of the state-closed abelian group $\Delta$, which simply corresponds to the group induced by addition of all possible power series in $\Z_d[[t]]$.

In all mentioned representations of the lamplighter groups on the rooted tree the generator of $\Z$ normalizes not only the base group, but the whole group $\Aff_I(X^*)$. In order to see what other lamplighter type groups of this kind can be realized by automata, it is natural to ask what is the structure of the normalizer of $\Aff_I(X^*)$ in $\Aut(X^*)$. The class of \emph{affine automorphisms} naturally arises from this question.

The affine automorphisms are the automorphisms of $X^*$ induced by affine transformations of the boundary of $X^*$ viewed as an infinite dimensional (uncountable) vector space $\Z_d^{\infty}$. More precisely, for each upper triangular matrix $A$ over $\Z_d$ with units along the main diagonal and each vector $\bb\in\Z_d^{\infty}$, we define an affine automorphism $\pi_{A,\bb}$ induced by a transformation
\[\pi_{A,\bb}(\bx)=\bb+\bx\cdot A\]
for each $\bx\in\Z_d^\infty$. All automorphisms from this class acting on $X^*$ form a group that we will denote by $\Aff(X^*)$. One of the main theorems of the paper is:
\begin{theoremnormalizer}
The normalizer of the group $\Aff_{I}(X^*)$ in $\Aut(X^*)$ coincides with the group $\Aff(X^*)$ of all affine automorphisms. In particular, in the case of binary tree, the normalizer of the group $\SHAut(\{0,1\}^*)$ of spherically homogeneous automorphisms in $\Aut(\{0,1\}^*)$ is $\Aff(\{0,1\}^*)$.
\end{theoremnormalizer}

%The automorphisms of $X^*$ induced by affine transformations of $\Z_d[[t]]$ (of the form $g(t)\mapsto b(t)+g(t)\cdot f(t)$) turn out to be affine automorphisms whose matrices are periodic with period one (i.e., the $i$-th row of such matrix starts with $i-1$ zeroes followed by $a_0,a_1,a_2,\ldots$, where $a_j$'s are coefficients of corresponding power series). Thus the class of affine automorphisms represent a natural generalization of the class of automorphisms induced by affine transformations of $\Z_d[[t]]$.\remDS{We can remove this paragraph if you would like to}

As our main motivating example, we completely describe the structure of a group $\G$ generated by a 4-state automaton shown in Figure~\ref{fig:automG}. This group has been considered by Klimann, Picantin, and the first author~\cite{klimann_ps:orbit_automata} where it was shown that it contains elements of infinite order. Initially this group was one of the 6 groups among those generated by 7471 non-minimally symmetric 4-state 2-letter automata, for which ``standard'' methods of finding elements of infinite order failed~\cite{caponi:thesis2014}. We prove

\begin{theoremG}
The group $\G=\langle a=(d,d)\sigma,b=(c,c),c=(a,b),d=(b,a)\rangle$ is isomorphic to the index 2 extension of the rank 2 lamplighter group:
\[G\cong \bigl(\Z_2^2\wr\Z\bigr)\rtimes\Z_2=\bigl(\langle x,y\rangle\wr\langle t\rangle\bigr)\rtimes\langle a\rangle,\]
where the action of $a$ on $x=ab,y=cd,t=ac$ is defined as follows: $x^a=x$, $y^a=y^{t^{-1}}$, $t^a=t^{-1}$. Moreover, $\G$ contains an index 2 subgroup consisting of affine automorphisms.
\end{theoremG}

We finally show that in the case of the lamplighter group $\Z_2\wr\Z$ acting faithfully on the binary tree, it is not a coincidence that we always see spherically homogeneous automorphisms. Namely, in Section~\ref{sec:state-closed_binary} we prove the following theorem.

\begin{theoremL}
Each state-closed faithful representations of the lamplighter group $\Z_2\wr\Z$ on the binary tree is conjugate to the one with the base group consisting of spherically homogeneous automorphisms.
\end{theoremL}

%the element $t=ac$ conjugating the base group in the rank 2 lamplighter, is an affine automorphism $\pi_{A,\bb}$ defined by a matrix $A$ of period 2 (whose $32\times32$-minor is shown in Figure~\ref{fig:matrix}) and a vector $\bb=[(1,0,0,1,1,1,0,0)^\infty]$. In particular, $t$ is not induced by the multiplication by power series in $\Z_2[[t]]$, thus this group represents an essentially new example of a realization of the lamplighter group.

%As is witnessed by the considered example, the group $\Aff(X^*)$ may contain interesting groups and we hope that the language of matrices that we introduce in Section~\ref{sec:affine} will turn out to be useful for their understanding.\remDS{We can remove this paragraph too.}

The structure of the paper is as follows. In Section~\ref{sec:prelim} we recall basic notions from the theory of automaton groups and set up terminology and notation. Section~\ref{sec:affine} introduces affine automorphisms and studies general properties of groups generated by them. In Section~\ref{sec:example} we describe in detail a nontrivial example of a group containing affine automorphisms and generating an index 2 extension of the rank 2 lamplighter group. Section~\ref{sec:state-closed_binary} studies the faithful state closed representations of the lamplighter group on the binary tree. Finally, we conclude the paper with several open questions listed in Section~\ref{sec:open}.

\noindent\textbf{Acknowledgement.} The first author is grateful to Rostislav Grigorchuk for fruitful discussions that stimulated the development of the paper.

\section{Preliminaries}
\label{sec:prelim}
We start from introducing the notions of tree automorphisms and Mealy automata. For more details we refer the reader to~\cite{gns00:automata}.

Let $X$ be a finite set of cardinality $d\geq 2$ and let $X^*$ denote the set of all finite words over $X$. This set can be naturally endowed with a structure of a rooted $d$-ary tree by declaring that $v$ is adjacent to $vx$ for any $v\in X^*$ and $x\in X$. The empty word corresponds to the root of the tree and $X^n$ corresponds to the $n$-th level of the tree. We will be interested in subgroups of the groups $\Aut(X^*)$ of all automorphisms of $X^*$ (as a graph). Any such automorphism can be defined via the notion of an initial Mealy automaton (possibly infinite).

\begin{definition}
A \emph{Mealy automaton} (or simply \emph{automaton}) is a tuple
$(Q,X,\pi,\lambda)$, where $Q$ is a set (the set of states), $X$ is a
finite alphabet, $\pi\colon Q\times X\to Q$ is the \emph{transition function}
and $\lambda\colon Q\times X\to X$ is the \emph{output function}. If the set
of states $Q$ is finite the automaton is called \emph{finite}. If
for every state $q\in Q$ the output function $\lambda(q,x)$ induces
a permutation of $X$, the automaton $\A$ is called \emph{invertible}.
Selecting a state $q\in Q$ produces an \emph{initial automaton}
$\A_q$.
\end{definition}

Automata are often represented by their \emph{Moore diagrams}. The
Moore diagram of an automaton $\A=(Q,X,\pi,\lambda)$ is a directed
graph in which the vertices are the states from $Q$ and the edges
have form $q\stackrel{x|\lambda(q,x)}{\longrightarrow}\pi(q,x)$ for
$q\in Q$ and $x\in X$. An example of a Moore diagram is shown in
Figure~\ref{fig:automG}.

Any invertible initial automaton $\A_q$ induces an automorphism of $X^*$ defined as follows. Given a word
$v=x_1x_2x_3\ldots x_n\in X^*$ it scans its first letter $x_1$ and
outputs $\lambda(q,x_1)$. The rest of the word is handled in a similar
fashion by the initial automaton $\A_{\pi(q,x_1)}$. Formally speaking,
the functions $\pi$ and $\lambda$ can be extended to $\pi\colon
Q\times X^*\to Q$ and $\lambda\colon  Q\times X^*\to X^*$ via
\[\begin{array}{l}
\pi(q,x_1x_2\ldots x_n)=\pi(\pi(q,x_1),x_2x_3\ldots x_n),\\
\lambda(q,x_1x_2\ldots x_n)=\lambda(q,x_1)\lambda(\pi(q,x_1),x_2x_3\ldots x_n).\\
\end{array}
\]

Note that any automorphism of $X^*$ induces the action on the set $X^\infty$ of all infinite words over $X$ that is identified with the \emph{boundary} of the tree $X^*$ consisting of all infinite paths in the tree without backtracking initiating at the root. The boundary of the tree is homeomorphic to the Cantor set and $\Aut(X^*)$ embeds into the group of homeomorphisms of this set.

\begin{definition}
The group generated by all states of an automaton $\A$ viewed as automorphisms
of the rooted tree $X^*$ under the operation of composition is
called an \emph{automaton group}.% and denoted by $\mathds G(\A)$.
\end{definition}

Note, that we do not require in the definition an automaton $\A$ to be finite. With this convention, the above notion is equivalent to the notions of \emph{self-similar group}~\cite{nekrash:self-similar} and \emph{state-closed group}~\cite{nekrash_s:12endomorph}. However, most of the interesting examples of automaton groups are finitely generated groups defined by finite automata.

Conversely, any automorphism of $X^*$ can be encoded by the action
of an invertible initial automaton. In order to show this we will need a notion of a
\emph{state} (often also called \emph{section}) of an automorphism at a vertex of the tree. Let $g$ be
an automorphism of the tree $X^*$ and $x\in X$. Then for any $v\in
X^*$ we have
\[g(xv)=g(x)v'\]
for some $v'\in X^*$. Then the map $g|_x\colon X^*\to X^*$ given by
\[g|_x(v)=v'\]
defines an automorphism of $X^*$ called the \emph{state} of
$g$ at vertex $x$. Furthermore,  for any finite word $x_1x_2\ldots x_n\in X^*$
we define \[g|_{x_1x_2\ldots x_n}=g|_{x_1}|_{x_2}\ldots|_{x_n}.\]

Given an automorphism $g$ of $X^*$ we construct an invertible initial automaton
$\A(g)$ whose action on $X^*$ coincides with that of $g$ as follows.
The set of states of $\A(g)$ is the set $\{g|_v\colon  v\in X^*\}$
of different states of $g$ at the vertices of the tree. The
transition and output functions are defined by
\[\begin{array}{l}
\pi(g|_v,x)=g|_{vx},\\
\lambda(g|_v,x)=g|_v(x).
\end{array}\]

Throughout the paper we will use the following convention. If $g$
and $h$ are the elements of some (semi)group acting on set $Y$ and
$y\in Y$, then
\begin{equation}
\label{eqn_conv}
gh(y)=h(g(y)).
\end{equation}

Taking into account convention~\eqref{eqn_conv} one can compute
the states of any element of an automaton semigroup as follows. If
$g=g_1g_2\cdots g_n$ and $v\in X^*$, then

\begin{equation}
\label{eqn_states} g|_v=g_1|_v\cdot g_2|_{g_1(v)}\cdots
g_n|_{g_1g_2\cdots g_{n-1}(v)}.
\end{equation}

For any automaton group $G$ there is a natural embedding
\[G\hookrightarrow G \wr \Sym(X)\]
defined by
\[G\ni g\mapsto (g_1,g_2,\ldots,g_d)\sigma(g)\in G\wr \Sym(X),\]
where $g_1,g_2,\ldots,g_d$ are the states of $g$ at the vertices
of the first level, and $\sigma(g)$ is a permutation of $X$ induced by the action of $g$ on the first level of the tree.

The above embedding is convenient in computations involving the
states of automorphisms, as well as for defining automaton groups. Sometimes it is called the \emph{wreath recursion} defining the group.

We conclude this section by a short discussion regarding spherically homogeneous automorphisms of $X^*$.

\begin{definition}
An automorphism $g$ of the tree $X^*$ is called \emph{spherically homogeneous} if for each level $l$ the states of $g$ at all vertices of $X^l$ act identically on the first level.
\end{definition}

It is a trivial observation that an automorphism $g$ is spherically homogeneous if and only if for all words $u,v\in X^*$ of the same length, $g|_u=g|_v$. For example, automorphisms $a=(a,a)\sigma,b=(a,a)$ are spherically homogeneous automorphisms of the binary tree. Every spherically homogeneous automorphism can be defined by a sequence $\{\sigma_n\}_{n\geq 1}$ of permutations of $X$ where $\sigma_n$ describes the action of $g$ on the $n$-th letter of the input word over $X$. Given a sequence $\{\sigma_n\}_{n\geq 1}$ we will denote the corresponding spherically homogeneous automorphism by $[\sigma_n]_{n\geq 1}$ or simply as $[\sigma_1,\sigma_2,\sigma_3,\ldots]$.

Obviously, all spherically homogeneous automorphisms of $X^*$ form a group, which we denote by $\SHAut(X^*)$, isomorphic to a product of uncountably many copies of $\Sym(|X|)$. In the case of the binary tree, this group is abelian and isomorphic to the abelianization of $\Aut(T_2)$.

Note, that for a finite state automorphism $g$ of $X^*$ it is algorithmically decidable whether $g$ is spherically homogeneous. First one checks if all states of $g$ at the vertices of the first level coincide. If this is not the case, then $g$ is not in $\SHAut(X^*)$. Otherwise, we repeat the procedure for the state $g|_x$, $x\in X$ (that does not depend on $x$). Since $g$ is finite state, this procedure will eventually terminate.

\section{Affine automorphisms of the tree}
\label{sec:affine}

A finite alphabet $X$ of cardinality $d$ can be endowed with the structure of the cyclic group $\Z_d$ of size $d$, and the boundary of the tree can be naturally identified with an infinite dimensional vector space $\Z_d^\infty$. After fixing a natural basis consisting of vectors $\be_i=[0,0,\ldots,0,1,0,\ldots],i\geq 1$ with the 1 at position $i$, elements of $\Z_d^\infty$ can be represented by infinite row vectors. We will consider automorphisms of $X^*$ that are induced by affine transformations of the boundary under the above identification.

Let $A$ be an infinite upper triangular matrix with entries from $\Z_d$ whose diagonal entries are units in $\Z_d$. We will denote the set of all such matrices by $U_{\infty}\Z_d$ (note, that finite dimensional upper triangular matrices describing certain automorphisms of rooted trees were used in a different context in~\cite{bartholdi_s:bsolitar}). Let also $\bb\in\Z_d^\infty$ be a row vector. Define the transformation $\pi_{A,\bb}\colon \Z_d^\infty\to\Z_d^\infty$ by
\[\pi_{A,\bb}(\bx)=\bb+\bx\cdot A.\]
Note that since $A$ is upper triangular, $\pi_{A,\bb}(\bx)$ is always well-defined.

\begin{proposition}
For each matrix $A\in U_{\infty}\Z_d$ and vector $\bb\in\Z_d^\infty$, the transformation $\pi_{A,\bb}$ induces an automorphism of the tree $X^*$.
\end{proposition}

\begin{proof}
Since by construction the matrix $A$ is invertible over $\Z_d$, the transformation $\pi_{A,\bb}$ is a bijection. Moreover, the triangular form of the matrix guarantees that the incidence relation in $X^*$ is preserved.
\end{proof}

With a slight abuse of notation we will denote the induced by $\pi_{A,\bb}$ automorphism of $X^*$ also by $\pi_{A,\bb}$. Recall that according to our convention of the right action, in the composition $gh$ of automorphisms of $X^*$ the automorphism $g$ acts first. With this convention in mind, we first list trivial properties of affine automorphisms that follow from corresponding properties of affine transformations of $\Z_d^\infty$:

\begin{proposition}
\label{prop:affine_basic} Let $A, A'$ be matrices over $\Z_d$, and $\bb,\bb'$ be vectors in $\Z_d^\infty$.
\begin{itemize}
\item[(a)] Automorphisms $\pi_{A,\bb}$ and $\pi_{A',\bb'}$ of $X^*$ coincide if and only if $A=A'$ and $b=b'$.
\item[(b)] $\pi_{A,\bb}\cdot\pi_{A',\bb'}=\pi_{AA',\bb A'+\bb'}$
\item[(c)] $\pi_{A,\bb}^{-1}=\pi_{A^{-1},-\bb A^{-1}}$
\end{itemize}
\end{proposition}

The last proposition guarantees that the set
\[\Aff(X^*)=\{\pi_{A,\bb}\in\Aut(X^*)\mid A\in U_\infty\Z_d,\ \bb\in\Z_d^\infty\}\]
forms a group, that we will call the \emph{group of affine automorphisms of} $X^*$.

The class of affine automorphisms is a natural generalization of the class of automorphisms induced by affine actions on the ring $\Z_d[[t]]$ of formal power series with coefficients in $\Z_d$. Namely, for each pair of power series $f(t),b(t)\in\Z_d[[t]]$ with $f(t)$ invertible (i.e. the coefficient at $t^0$ is a unit in $\Z_d$) one can define an \emph{affine} transformation $\tau_{f,b}$ of $\Z_d[[t]]$ by
\[\bigl(\tau_{f,b}(g)\bigr)(t)=b(t)+g(t)\cdot f(t).\]
Under the natural identification of $\Z_d[[t]]$ with the boundary of $X^*$, the transformation $\tau_{f,b}$ induces an automorphism of $X^*$, that we will also denote by $\tau_{f,b}$ where the context is clear. Such automorphisms have been studied in the contexts of lamplighter groups~\cite{gns00:automata} and Cayley machines~\cite{silva_s:lamplighter05}. For example, the standard automaton representation of the lamplighter group $\Z_2\wr\Z$ on the binary tree is obtained from the automaton defining $\tau_{1+t,0}$.

The proof of the following proposition is straightforward.

\begin{proposition}
Let $f(t)=a_0+a_1t+a_2t^2+\ldots$ and $b(t)=b_0+b_1t+b_2t^2+\ldots$ be power series in $\Z_d[[t]]$ with $a_0$ being a unit. The automorphism $\tau_{f,b}$ coincides with the affine automorphism $\pi_{A,\bb}$ for $\bb=[b_0,b_1,b_2,\ldots]$ and
\[A=\left[\begin{array}{ccccc}
a_{0}&a_{1}&a_{2}&a_3&\ldots\\
0&a_{0}&a_{1}&a_2&\ldots\\
0&0&a_{0}&a_1&\ldots\\
0&0&0&a_{0}&\ldots\\
\vdots&\vdots&\vdots&\vdots&\ddots
\end{array}\right]\]
whose $i$-th row starts with $i-1$ zeros followed by $a_0,a_1,a_2,\ldots$.
\end{proposition}

Similarly to affine automorphisms of $X^*$, the set
\[\Aff_{[[t]]}(X^*)=\{\tau_{f,b}\in\Aut(X^*)\mid f,b\in\Z_d[[t]]\}\]
also forms a group (a proper subgroup of $\Aff(X^*)$) that we will call the \emph{group of $\Z_d[[t]]$-affine automorphisms of} $X^*$.

To address the question on when an affine automorphism is finite state, we set up the notation for the shift map defined for matrices and vectors.
\begin{definition}
\begin{itemize}
\item The \emph{shift} of a vector $\bb=[b_{i}]_{i\geq 1}$ is the vector $\sigma(\bb)=[b_{i}]_{i\geq 2}$ obtained from $\bb$ by removing the first entry.
\item The \emph{shift} of a matrix $A=[a_{ij}]_{i,j\geq 1}$ is the matrix $\sigma(A)=[a_{ij}]_{i,j\geq 2}$ obtained from $A$ by removing the first row and the first column.
\end{itemize}
\end{definition}

We will say that a vector $\bb$ (resp., a matrix $A$) is \emph{eventually periodic} if $\{\sigma^n(\bb),n\geq0\}$ (resp., $\{\sigma^n(A),n\geq0\}$) is finite.

\begin{proposition}
\label{prop:state}
Let $\pi_{A,\bb}$ be an affine automorphism, and let $x\in X=\Z_d$ be a letter. Then the state of $\pi_{A,\bb}$ at vertex $x$ of $X^*$ is
\begin{equation}
\label{eqn:state}
\pi_{A,\bb}|_x=\pi_{\sigma(A),x\cdot\sigma([1,0,0,\ldots]A)+\sigma(\bb)}.
\end{equation}
\end{proposition}

\begin{proof}
Let
\[A=\left[\begin{array}{cccc}
a_{11}&a_{12}&a_{13}&\ldots\\
0&a_{22}&a_{23}&\ldots\\
0&0&a_{33}&\ldots\\
\vdots&\vdots&\vdots&\ddots
\end{array}\right],\qquad
\bb=[b_1,b_2,b_3,\ldots]
\]
be the matrix and the vector defining $\pi_{A,\bb}$. Let also $\bx=[x_1,x_2,x_3,\ldots]$ be an arbitrary point in the boundary of $X^*$. Then
\begin{multline*}
\pi_{A,\bb}(\bx)=\bb+\bx\cdot A=[b_1+a_{11}x_1,\ b_2+a_{12}x_1+a_{22}x_2,\ b_3+a_{13}x_1+a_{23}x_2+a_{33}x_3,\ldots]=\\
[b_1+a_{11}x_1,\  [b_2,b_3,\ldots]+x_1\cdot[a_{12},a_{13},\ldots]+[x_2,x_3,\ldots]\cdot\sigma(A)].
\end{multline*}

Since the image of $[x_2,x_3,\ldots]$ under $\pi_{A,\bb}|_{x_1}$ is obtained by erasing the first entry in the above vector, we immediately obtain~\eqref{eqn:state}.
\end{proof}

\begin{theorem}
The automorphism $\pi_{A,\bb}$ is finite state if and only if matrix $A$, its rows, and vector $\bb$ are eventually periodic.
\end{theorem}

\begin{proof}
The ``if'' direction follows immediately from Proposition~\ref{prop:state}. Thus we concentrate on the ``only if'' direction. Suppose an automorphism $\pi_{A,\bb}$ is finite state. Then, in particular, the set $\{\pi_{A,\bb}|_{0^n},\ n\geq0\}$ is finite. By Proposition~\ref{prop:state} we have
\[\{\pi_{A,\bb}|_{0^n},\ n\geq0\}=\{\pi_{\sigma^n(A),\sigma^n(\bb)},\ n\geq0\}.\]
Therefore by Lemma~\ref{prop:affine_basic}(a) we must have
\begin{equation}
\label{eqn:Ab_finite}
\{\sigma^n(A),\ n\geq0\},\quad \{\sigma^n(\bb),\ n\geq0\}\qquad\text{are finite}
\end{equation}
and both $A$ and $\bb$ are eventually periodic.

To prove that the $j$-th row $\ba_j$ of $A$ is also eventually periodic we note that the set
\[\{\pi_{A,\bb}|_{0^{j-1}10^n},\ j\geq1,n\geq0\}=\{\pi_{\sigma^{j-1}(A),\sigma^{j-1}(\bb)}|_{10^n},\ j\geq1,n\geq0\}=\{\pi_{\sigma^{j+n}(A),\sigma^{1+n}(\ba_{j})+\sigma^{1+n}(\bb)}\}\]
must be finite (where in the last equality we used Proposition~\ref{prop:state}). Since we have proved in~\eqref{eqn:Ab_finite} that $A$ and $\bb$ are eventually periodic, we must have that $\{\sigma^{1+n}(\ba_{j}),\ n\geq1\}$ is also finite and thus $\ba_j$ is eventually periodic.
\end{proof}

As a partial case we obtain a known results regarding the automorphisms induced by the affine transformations of $\Z_d[[t]]$. We will use below the following natural notation: for $c(t)=c_0+c_1t+c_2t^2+\ldots\in\Z_d[[t]]$ we define the \emph{shift} of $c(t)$ by \[\sigma(c(t))=c_1+c_2t+c_3t^2+\ldots=\frac{c(t)-c_0}{t}.\]

\begin{corollary}
Let $f(t),b(t)\in\Z_d[[t]]$ be arbitrary power series with $f(t)$ invertible. Then
\begin{itemize}
\item[(a)] the automorphism $\tau_{f,b}$ is finite state if and only if both $f(t)$ and $b(t)$ are rational power series.
\item[(b)] for each $x\in X$ the state of $\tau_{f,b}$ at $x$ is computed as
\[\tau_{f,b}|_x=\tau_{f,x\sigma(f)+\sigma(b)}.\]
\end{itemize}
\end{corollary}

\begin{corollary}
Let $\pi_{A,b}$ be an affine automorphism. The period of the matrix $A$ is a factor of the length of the shortest cycle in the automaton defining $\pi_{A,b}$. In particular, an affine automorphism $\pi_{A,b}$ is induced by an affine transformation of $\Z_d[[t]]$ if and only if there is a loop of length one in the automaton defining $\pi_{A,b}$.
\end{corollary}

\begin{proof}
Follows immediately from Proposition~\ref{prop:state}. If $g=\pi_{A,b}$ and $g|_u=g|_{uv}$ are the two states in the automaton defining $g$ that correspond to the shortest cycle of length $|v|$, then $\pi_{\sigma^{|u|}(A),*}=g|_u=g|_v=\pi_{\sigma^{|uv|}(A),*}$. Therefore $\sigma^{|uv|}(A)=\sigma^{|u|}(A)$ by Proposition~\ref{prop:affine_basic}(a).
\end{proof}

Let $I$ denote the infinite identity matrix over $\Z_d$. Then the group $\SHAut(X^*)$ of spherically homogeneous automorphisms of $X^*$ naturally contains (and in the case $d=2$ is equal to) the group $\Aff_{I}(X^*)=\{\pi_{I,\bb}, \bb\in\Z_d^\infty\}$ of affine shifts. Elements of $\Aff_{I}(X^*)$ can be parameterized by infinite-dimensional vectors from $\Z_d^\infty$, where the $i$-th component defines the action of an automorphism on the $i$-th letter of an input word.

\begin{proposition}
\label{prop:metab_general}
The group $\Aff_{I}(X^*)$ is a normal subgroup of $\Aff(X^*)$. In particular, for each abelian subgroup $U$ of $\Aff(X^*)$, the group $\langle U, \Aff_{I}(X^*)\rangle$ is metabelian.
\end{proposition}

\begin{proof}
Let $\pi_{I,\bb}\in\Aff_{I}(X^*)$ and $\pi_{A,\bc}\in\Aff(X^*)$ be arbitrary elements. Using Proposition~\ref{prop:affine_basic} we calculate:
\begin{multline*}
\pi_{I,\bb}^{\pi_{A,\bc}}=(\pi_{A,\bc}^{-1}\cdot\pi_{I,\bb})\cdot\pi_{A,\bc}=(\pi_{A^{-1},-\bc A^{-1}}\cdot\pi_{I,\bb})\cdot\pi_{A,\bc}\\
=\pi_{A^{-1},-\bc A^{-1}I+\bb}\cdot\pi_{A,\bc}=\pi_{I,(-\bc A^{-1}+\bb)A+\bc}=\pi_{I,\bb A}\in\Aff_I(X^*).
\end{multline*}
\end{proof}

\begin{corollary}
The group $\Aff_{[[t]]}(X^*)$ is metabelian.
\end{corollary}

\begin{proof}
It is enough to apply Proposition~\ref{prop:metab_general} in the case when $U=\{\tau_{f,0}\mid f\in\Z_d[[t]]\}$. Since $\tau_{f,g}=\tau_{f,0}\cdot\tau_{0,g}$, we get that $\Aff_{[[t]]}(X^*)$ is a subgroup of a metabelian group $\langle U, \Aff_{I}(X^*)\rangle$, and thus is metabelian.
\end{proof}

\begin{lemma}
\label{lem:centralizer}
The centralizer $C$ of the group $\Aff_{I}(X^*)$ in $\Aut(X^*)$ coincides with $\Aff_{I}(X^*)$.
\end{lemma}

\begin{proof}
Since $\Aff_{I}(X^*)$ is abelian, it is a subgroup of $C$. Suppose there is an element $h\in C-\Aff_{I}(X^*)$. Let $n\geq 0$ be the smallest integer such the action of $h$ on the $(n+1)$-st level does not coincide with corresponding action of any element of $\Aff_{I}(X^*)$. There are two possible cases.

\noindent \textbf{Case I.} There are two distinct vertices $v$ and $v'$ of the $n$-th level of $X^*$ such that the permutations $\lambda$ and $\lambda'$ of $X$ induced by the actions of $h|_v$ and $h|_{v'}$, respectively, on the first level of $X^*$ are different. By the minimality of $n$, there is an element $s\in\Aff_{I}(X^*)$ such that $h_1=s^{-1}h\in C-\Aff_{I}(X^*)$ fixes the $n$-th level of $X^*$.

Let $g\in\Aff_{I}(X^*)$ be an automorphism mapping $v$ to $v'$ and not changing the $(n+1)$-st letter of an input word. Let also $x\in X$ be such that $\lambda(x)\neq\lambda'(x)$. Then we have
\[h_1(g(vx))=h_1(v'x)=v'\lambda'(x)\neq v'\lambda(x)=g(v\lambda(x))=g(h_1(vx)),\]
contradicting to the fact that $h_1\in C$.

\noindent \textbf{Case II.} For all vertices $v$ of level $n$ elements $h|_v$ induce the same permutation $\lambda$ of $X$, but $\lambda$ is not in $\Z_d=\langle(0,1,2,\ldots,d-1)\rangle$. In this case, since $\langle(0,1,2,\ldots,d-1)\rangle$ is self-centralized in $\Sym(X)$, there is
an element $\lambda_1\in\langle(0,1,2,\ldots,d-1)\rangle$ not commuting with $\lambda$. Any element of $\Aff_{I}(X^*)$ acting on the $(n+1)$-st letter of an input word as $\lambda_1$ will not commute with $h$, again contradicting to the fact that $h\in C$.
\end{proof}

\begin{theorem}
\label{thm:normalizer}
The normalizer $N$ of the group $\Aff_{I}(X^*)$ in $\Aut(X^*)$ coincides with the group $\Aff(X^*)$ of all affine automorphisms. In particular, in the case of the binary tree, the normalizer of $\SHAut(\{0,1\}^*)$ in $\Aut(\{0,1\}^*)$ is $\Aff(\{0,1\}^*)$.
\end{theorem}

\begin{proof}
By Lemma~\ref{lem:centralizer} and the N/C theorem there is a homomorphism $\phi$ from $N$ onto a subgroup of $\Aut(C)$ whose kernel is $C$, defined by $\bigl(\phi(g)\bigr)(h)=g^{-1}hg$. We will denote $\phi(g)$ by $\phi_g$ for simplicity.

Since $C=\Aff_{I}(X^*)$ is isomorphic to $\Z_d^\infty$, each automorphism $\alpha$ of $C$ is defined by a matrix $A_{\alpha}$ invertible over $\Z_d$, which sends a vector $\bb\in\Z_d^\infty$ to $\bb\cdot A$. The $i$-th row of the matrix $A$ corresponds to $\alpha(\sigma^{(i-1)})$.

Suppose now that $g\in N$ is an arbitrary element of $N$. Then since $\sigma^{(n)}$ fixes the $n$-th level of $X^*$, $\phi_g(\sigma^{(n)})=g^{-1}\sigma^{(n)}g$ also must fix the $n$-th level. Therefore, the $(n+1)$-st row of matrix $A_{\phi_g}$ starts with $n$ zeroes, i.e. $A_{\phi_g}$ is upper triangular. At the same time the invertibility of $A_{\phi_g}$ guarantees that $A_{\phi_g}\in U_{\infty}\Z_d$. Moreover, for each $A\in U_{\infty}\Z_d$ by the proof of Proposition~\ref{prop:metab_general}, there is $g=\pi_{A,0}\in N$ such that $A=A_{\phi_g}$.

Finally, since conjugation by elements of the form $\pi_{A,0}$ gives all possible conjugations by elements of $N$, and this set forms a group by Proposition~\ref{prop:affine_basic}, we deduce that
\[N=\Aff_{I}(X^*)\rtimes \{\pi_{A,0}\mid A\in U_\infty\Z_d\},\]
but the last set coincides with $\Aff(X^*)$.
\end{proof}

\section{The principal example}
\label{sec:example}

In this section we give a complete description of the structure of a group $\G$ first mentioned in~\cite{caponi:thesis2014} and studied in~\cite{klimann_ps:orbit_automata}. This group was initially one of the six groups among those generated by 7421 non-minimally symmetric 4-state invertible automata over 2-letter alphabet studied in~\cite{caponi:thesis2014}, for which the existence of elements of infinite order could not be established by standard known methods implemented in~\cite{muntyan_s:automgrp}. In~\cite{klimann_ps:orbit_automata} many such elements were found using a new technique of orbit automata. However, the complete structure of this group was yet to be understood. Below, we develop a new technique to work with this and similar groups that allows us to answer this question completely. In particular, we show that this group contains an index 2 subgroup consisting of affine automorphisms.

The group $\G$ is generated by the 4-state automaton depicted in Figure~\ref{fig:automG} with the following wreath recursion:
\begin{equation}
\label{eqn_autom_def}
\begin{array}{lcl}
a&=&(d,d)\sigma,\\
b&=&(c,c),\\
c&=&(a,b),\\
d&=&(b,a).
\end{array}
\end{equation}
 \begin{figure}[!h]
         \begin{center}
         \includegraphics{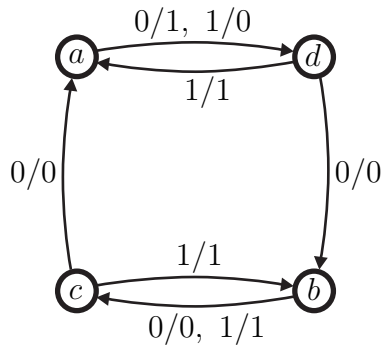}
         \end{center}
 \caption{Automaton $\mathcal A$ generating the group~$\G$.\label{fig:automG}}
 \end{figure}

All generators of~$\G$ have order 2 and the subgroups $\langle a,b\rangle$ and $\langle c,d\rangle$ are both isomorphic to the 4-element Klein group~$\Z_2^2$. We can also rewrite the definition of the group $\G$ in the following form (using the notation introduced before Theorem~\ref{thm:normalizer}):
\[a=d^{(1)}\sigma,\quad b=c^{(1)},\quad c=(a,b),\quad d=(b,a).\]

Consider the following elements of $\G$:
\[x:=ab,\quad y=cd,\quad t=ac.\]
Since $x=y^{(1)}\sigma$ and $y=x^{(1)}$, we immediately get that both $x$ and $y$ are spherically homogeneous. Moreover, it is straightforward to check that $\langle x,y\rangle$ is also isomorphic to $\Z_2^2$. Moreover, the elements $t$ and $t^{-1}$ have the following decompositions:
\begin{eqnarray}
\label{eqn:t}t&=&(x^ty,y)(t^{-1})^{(1)}\sigma\\
\label{eqn:t_t_inv}t^{-1}&=&(y^{t^{-1}},xy^{t^{-1}})t^{(1)}\sigma\label{eqn:t_inv}
\end{eqnarray}

\begin{lemma}
\label{lem:xtyt}
The elements $x^t$, $x^{t^{-1}}$, $y^t$, $y^{t^{-1}}$ are spherically homogeneous.
\end{lemma}

\begin{proof}
According to equations~\eqref{eqn:t} and~\eqref{eqn:t_t_inv} we have:
\begin{multline*}
x^t=\left(y^{t^{-1}},xy^{t^{-1}}\right)t^{(1)}\sigma\cdot (y,y)\sigma\cdot(x^ty,y)(t^{-1})^{(1)}\sigma\\=\left(y^{t^{-1}}tyx^tyt^{-1},xy^{t^{-1}}tyyt^{-1}\right)\sigma=\left(xy^{t^{-1}},xy^{t^{-1}}\right)\sigma.
\end{multline*}
Since $x$ is spherically homogeneous we need only to show that $y^{t^{-1}}$ is also spherically homogeneous. Indeed,
\begin{multline*}
y^{t^{-1}}=(x^ty,y)(t^{-1})^{(1)}\sigma\cdot (x,x)\cdot\left(y^{t^{-1}},xy^{t^{-1}}\right)t^{(1)}\sigma\\=
\left(x^tyt^{-1}xxy^{t^{-1}}t,yt^{-1}xy^{t^{-1}}t\right)\sigma
=\left(x^t,x^t\right)
\end{multline*}
so we again obtain $x^t$, which implies that both $x^t$ and $y^{t^{-1}}$ are spherically homogeneous. Similar argument proves that $x^{t^{-1}}$ and $y^t$ are also spherically homogeneous.
\end{proof}

\begin{lemma}
\label{lem:t_in_norm}
The automorphism $t$ lies in the normalizer of the group $\SHAut(X^*)$.
\end{lemma}

\begin{proof}
It is enough to prove that $(\sigma^{(n)})^t,(\sigma^{(n)})^{t^{-1}}\in\SHAut(X^*)$. We prove this by induction on $n$.

First we verify the induction base for $n=0$ and $\sigma^{(0)}=(1,1)\sigma$. Using expressions~\eqref{eqn:t} and~\eqref{eqn:t_t_inv} we calculate:
\begin{multline*}
(\sigma^{(0)})^t=\left(y^{t^{-1}},xy^{t^{-1}}\right)t^{(1)}\sigma\cdot (1,1)\sigma\cdot(x^ty,y)(t^{-1})^{(1)}\sigma\\=\left(y^{t^{-1}}tx^tyt^{-1},xy^{t^{-1}}tyt^{-1}\right)\sigma
=\left(y^{t^{-1}}xy^{t^{-1}},x\bigl(y^{t^{-1}}\bigr)^2\right)\sigma=(x,x)\sigma.
\end{multline*}
Similarly
\begin{multline*}
(\sigma^{(0)})^{t^{-1}}=(x^ty,y)(t^{-1})^{(1)}\sigma\cdot (1,1)\sigma\cdot\left(y^{t^{-1}},xy^{t^{-1}}\right)t^{(1)}\sigma\\=
\left(x^tyt^{-1}y^{t^{-1}}t,yt^{-1}xy^{t^{-1}}t\right)\sigma
=\left(x^ty^2,y^2x^t\right)\sigma=(x^t,x^t)\sigma,
\end{multline*}
which shows that $(\sigma^{(0)})^{t^{-1}}$ is spherically homogeneous by Lemma~\ref{lem:xtyt}.

To prove the induction step, we assume that $(\sigma^{(n)})^t,(\sigma^{(n)})^{t^{-1}}\in\SHAut(X^*)$ for some $n\geq 1$. Then
\begin{multline*}
(\sigma^{(n+1)})^t=\left(y^{t^{-1}},xy^{t^{-1}}\right)t^{(1)}\sigma\cdot (\sigma^{(n)},\sigma^{(n)})\cdot(x^ty,y)(t^{-1})^{(1)}\sigma\\
=\left(y^{t^{-1}}t\sigma^{(n)}yt^{-1},xy^{t^{-1}}t\sigma^{(n)}x^tyt^{-1}\right)
=\left(y^{t^{-1}}(\sigma^{(n)})^{t^{-1}}y^{t^{-1}},xy^{t^{-1}}(\sigma^{(n)})^{t^{-1}}xy^{t^{-1}}\right)\\=\left((\sigma^{(n)})^{t^{-1}},(\sigma^{(n)})^{t^{-1}}\right)
\end{multline*}
is again spherically homogeneous by the inductive assumption and Lemma~\ref{lem:xtyt}. Similarly one can show that $(\sigma^{(n+1)})^{t^{-1}}$ is also in $\SHAut(X^*)$.
\end{proof}

As a direct corollary of the previous lemma and Theorem~\ref{thm:normalizer} we obtain that $t$ is an affine automorphism and, hence, $t=\pi_{A,\bb}$ for some $A$ and $\bb$. More precisely, we get:

\begin{corollary}
The automorphism $t$ is equal to $\pi_{A,\bb}$ for the matrix $A$ with the row $2i-1$ (resp., row $2i$) of the form $[0^{2i-2},1,(1,0)^\infty]$ (resp., $[0^{2i-1},1,(1,1,1,0)^\infty]$) for $i\geq 1$ (see Figure~\ref{fig:matrix}), and for $\bb=[(1,0,0,1,1,1,0,0)^\infty]$.
\end{corollary}

\begin{figure}[!h]
\begin{center}
\includegraphics[width=200pt]{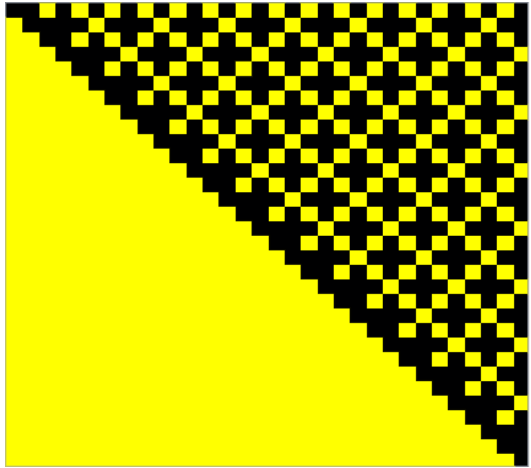}
\end{center}
\caption{$32\times32$-minor of matrix $A$ involved in the definition of $t=ac$, where the black squares indicate 1's.\label{fig:matrix}}
\end{figure}

\begin{proof}
We find $\bb$ simply by computing $\bb=\bb+[0,0,0,\ldots]\cdot A=\pi_{A,\bb}(0^\infty)=t(0^\infty)=[(1,0,0,1,1,1,0,0)^\infty]$. With the knowledge of $\bb$ we can compute the $i$-th row $\ba_i$ of matrix $A$. Let $\be_i=[0,0,\ldots,0,1,0,\ldots]$ be the $i$-th standard basis vector in $\Z_2^\infty$. Then $t(\be_i)=\bb+\be_i\cdot A=\bb+\ba_i$ and, thus,
\begin{equation}
\label{eqn:rows_of_A}
\ba_i=t(\be_i)-\bb.
\end{equation}
Moreover, since $t|_{01}=t$, by Proposition~\ref{prop:state} we obtain
\[\pi_{A,\bb}=\pi_{A,\bb}|_{01}=\pi_{\sigma^2(A),b'}\]
for some vector $b'$. Therefore, by Proposition~\ref{prop:affine_basic}(a), $\sigma^2(A)=A$ and we only need to compute the first two rows of $A$. A direct computation using~\eqref{eqn:rows_of_A} twice yields:
\[\begin{array}{ll}
\ba_1=[1,(1,0)^\infty]\\
\ba_2=[0,1,(1,1,1,0)^\infty].
\end{array}\]
\end{proof}

By Lemma~\ref{lem:t_in_norm}, conjugates of any spherically homogeneous element $z\in\SHAut(X^*)$ by powers (possibly negative) of $t$ are also spherically homogeneous, and hence commute. Therefore the following convenient notation is well-defined for $i_j\in\Z$.
\begin{equation}
\label{eqn:notation_exp}
z^{t^{i_1}+t^{i_2}+\cdots+t^{i_n}}:=z^{t^{i_1}}z^{t^{i_2}}\cdots z^{t^{i_n}}.
\end{equation}
In particular, for each Laurent polynomial $p(t)\in\Z_2[t,t^{-1}]$ the elements $x^{p(t)}$ and $y^{p(t)}$ are defined. In order to show that $\langle x,y,t\rangle$ is isomorphic to $\Z_2^2\wr\Z$, it is enough to show that for each pair of Laurent polynomials $p(t)$, $q(t)$ the element $x^{p(t)}y^{q(t)}$ is not trivial. Note, that the idea of translating the base group in the lamplighter groups as powers of the ring of Laurent polynomials have been used in~\cite{grigorch_k:lamplighter,sidki_d:lamplighter}.

\begin{lemma}
\label{lem:xpq}
Let $p(t)$ and $q(t)$ be two Laurent polynomials. Then
\begin{eqnarray}
\label{eqn:pqt}
\left((x^{p(t)}y^{q(t)})^{(1)}\sigma\right)^t&=&(x^{p(t)t^{-1}+1}y^{q(t)t^{-1}})^{(1)}\sigma\\
\left((x^{p(t)}y^{q(t)})^{(1)}\sigma\right)^{t^{-1}}&=&(x^{p(t)t+t}y^{q(t)t})^{(1)}\sigma\\
\left((x^{p(t)}y^{q(t)})^{(1)}\right)^t&=&(x^{p(t)t^{-1}}y^{q(t)t^{-1}})^{(1)}\\
\left((x^{p(t)}y^{q(t)})^{(1)}\right)^{t^{-1}}&=&(x^{p(t)t}y^{q(t)t})^{(1)}
\end{eqnarray}
\end{lemma}

\begin{proof}
We prove only equality~\eqref{eqn:pqt}. The proof of other equalities in the statement is almost identical. Using equality~\eqref{eqn:t}, the fact that conjugates of $x$ and $y$ by powers of $t$ commute, and that $x^2=y^2=1$, we compute:
\begin{multline*}
\left((x^{p(t)}y^{q(t)})^{(1)}\sigma\right)^t=\left((x^{p(t)}y^{q(t)})^{(1)}\sigma\right)^{(x^ty,y)(t^{-1})^{(1)}\sigma}=
\left((x^{p(t)}y^{q(t)})^{(1)}x^ty^2\sigma\right)^{(t^{-1})^{(1)}\sigma}\\=\left((x^{p(t)+t}y^{q(t)})^{(1)}\sigma\right)^{(t^{-1})^{(1)}\sigma}
=(x^{p(t)t^{-1}+1}y^{q(t)t^{-1}})^{(1)}\sigma.
\end{multline*}
\end{proof}

\begin{lemma}
\label{lem:xtn}
For each $n\geq 1$ we have the following equalities:
\begin{eqnarray}
x^{t^n}&=&\bigl(x^{1+t^{-1}+t^{-2}+\cdots+t^{-n+1}}\cdot y^{t^{-n}}\bigr)^{(1)}\sigma,\\
y^{t^n}&=&\bigl(x^{t^{-n}}\bigr)^{(1)},\\
x^{t^{-n}}&=&\bigl(x^{t+t^{2}+\cdots+t^{n}}\cdot y^{t^{n}}\bigr)^{(1)}\sigma,\\
y^{t^{-n}}&=&\bigl(x^{t^{n}}\bigr)^{(1)}.
\end{eqnarray}
\end{lemma}

\begin{proof}
Since $x^{t^0}=x=(x^0y^1)^{(1)}\sigma$ and $y^{t^0}=y=(x^1y^0)^{(1)}$, we immediately obtain the statement of the lemma by induction on $n$ from Lemma~\ref{lem:xpq}.
\end{proof}

For each $n\geq 1$ define
\[\phi_n(t):=1+t+t^2+\cdots+t^{n-1}.\]
For each polynomial $p(t)=\sum_{i=0}^ka_it^i\in\Z_2[t]$  define also
\[\psi_p(t)=\sum_{i=1}^ka_i\phi_i(t).\]
Note, that with a convention that $\deg 0=-1$, for each nonzero polynomial $p$:
\[\deg\psi_p=\deg p-1.\]

\begin{lemma}
\label{lem:xpyq_states}
For all pairs of polynomials $p(t),q(t)\in\Z_2[t]$
\begin{itemize}
\item the state of $x^{p(t)}y^{q(t)}$ at each vertex of the first level is $x^{\psi_p(t^{-1})+q(t^{-1})}y^{p(t^{-1})}$.
\item the state of $x^{p(t^{-1})}y^{q(t^{-1})}$ at each vertex of the first level is $x^{t\psi_p(t)+q(t)}y^{p(t)}$.
\end{itemize}
\end{lemma}

\begin{proof}
The result immediately follows from Lemma~\ref{lem:xtn} and the fact that the conjugates of $x$ and $y$ by powers of $t$ commute.
\end{proof}

To simplify notation, for polynomials $p(t),q(t)\in\Z_2[t]$ we denote $x^{p(t)}y^{q(t)}$ by $(p,q)^+$ and $x^{p(t^{-1})}y^{q(t^{-1})}$ by $(p,q)^-$. Also for a spherically homogeneous automorphism $g$ we will denote by $g\rightarrow g|_0$ the operation of passing to the state at a vertex of the first level (it is well defined because states at all vertices of the first level coincide). With these notations established, Lemma~\ref{lem:xpyq_states} can be reformulated as follows:
\begin{eqnarray}
\label{eqn:pq1}(p,q)^+&\rightarrow&(\psi_p+q,p)^-,\\
\label{eqn:pq2}(p,q)^-&\rightarrow&(t\psi_p+q,p)^+.
\end{eqnarray}

The next two lemmas constitute the technical heart of the proof of Theorem~\ref{thm:structure_of_G}.

\begin{lemma}
\label{lem:same_deg}
If there is a pair $(p(t),q(t))$ of polynomials in $\Z_2[t]$ such that $x^{p(t)}y^{q(t)}$ is trivial, then there is a pair $(p'(t),q'(t))$ of polynomials with $x^{p'(t)}y^{q'(t)}=1$ and $\deg p'=\deg q'=\max\{\deg p,\deg q\}$.
\end{lemma}

\begin{proof}
Suppose $(p,q)$ is such a pair, i.e. $x^{p(t)}y^{q(t)}=(p,q)^+=1$. It is straightforward to check that $x^{p(t)}y^{q(t)}\neq 1$ if $\max\{\deg p,\deg q\}<2$, so we can assume that $\max\{\deg p,\deg q\}\geq 2$. Consider three cases.

\noindent \textbf{Case I.} $\deg p<\deg q$.\\
In this case using~\eqref{eqn:pq1} and~\eqref{eqn:pq2} we calculate:
\[(p,q)^+\rightarrow(\psi_p+q,p)^-\rightarrow(t\cdot\psi_{\psi_p+q}+p,\psi_p+q)^+\]
Since for each nonzero polynomial $r$ we have $\deg\psi_r=\deg r-1$ and for zero polynomial $\deg\psi_0=\deg 0$, we have
\[\deg\psi_p\leq\deg p<\deg q\ \Rightarrow\ \deg(\psi_p+q)=\deg q\]
and therefore
\[\deg\psi_{\psi_p+q}=\deg q-1\ \Rightarrow\ \deg(t\cdot\psi_{\psi_p+q}+p)=\deg q.\]

\noindent \textbf{Case II.} $\deg p=\deg q+1\geq 2$.\\
First we observe that $t^a=(ac)^a=ca=t^{-1}$, $x^a=x$, and $y^a=y^{t^{-1}}$. Therefore, for each pair of polynomials $p(t),q(t)\in\Z_2[t]$ we have
\[\left(x^{p(t)}y^{q(t)}\right)^a=x^{p(t^{-1})}y^{t^{-1}q(t^{-1})}.\]
In other words, conjugation of $(p,q)^+$ by $a$ produces $(p,t\cdot q)^-$, which also corresponds to the identity in $\G$. Passing to the first level state now yields:
\[(p,t\cdot q)^-\rightarrow (t\cdot\psi_{p}+t\cdot q,p)^+,\]
where $\deg(t\cdot \psi_{p}+t\cdot q)<\deg p$ since both $t\cdot\psi_{p}$ and $t\cdot q$ are polynomials of degree $\deg p\geq 2$, so that the highest term cancels in their sum. Thus, we arrived to the previous case.

\noindent \textbf{Case III.} $\deg p>\deg q+1$.\\
In this case we pass to the state on the second level:
\[(p,q)^+\rightarrow (\psi_p+q,p)^-\rightarrow \left(t\cdot\psi_{\psi_p+q}+p,\psi_p+q\right)^+.\]
Now observe that $\deg(\psi_p+q)=\deg p-1$, so $\deg\psi_{\psi_p+q}=\deg p-2$ (since we assumed $\deg p\geq2$) and thus $\deg(t\cdot\psi_{\psi_p+q}+p)=\deg p$. Therefore we arrive to the situation described in Case II.
\end{proof}

\begin{lemma}
\label{lem:xpyq}
For each pair $(p(t),q(t))$ of polynomials in $\Z_2[t]$ the automorphism $x^{p(t)}y^{q(t)}$ is nontrivial.
\end{lemma}

\begin{proof}
Assume on the contrary that the statement is false. Choose a pair of polynomials $p,q$ such that $(p,q)^+$ corresponds to the identity in $\G$ and $\max\{\deg p,\deg q\}$ is the smallest. Moreover, by Lemma~\ref{lem:same_deg} without loss of generality we can assume that $p$ and $q$ are of the same degree that is greater than or equal 2. Passing to the second level state yields:
\begin{equation*}
(p,q)^+\rightarrow(\psi_p+q,p)^-\rightarrow \left(t\cdot\psi_{\psi_p+q}+p,\psi_p+q\right)^+.
\end{equation*}
Since both $t\cdot\psi_{\psi_p+q}$ and $p$ have degree $\deg p$, their sum $p_1:=t\cdot\psi_{\psi_p+q}+p$ has degree less than $\deg p$, while $\deg(\psi_p+q)=\deg p$. On the next level we obtain:
\begin{equation*}
\left(t\cdot\psi_{\psi_p+q}+p,\psi_p+q\right)^+\rightarrow \left(\psi_{p_1}+\psi_p+q,p_1\right)^-,
\end{equation*}
where for $p_2:=\psi_{p_1}+\psi_p+q$ we have $\deg p_2=\deg p$ (since degrees of $\psi_{p_1}$ and $\psi_{p}$ are less than $\deg q=\deg p$). Finally, after passing to the state on the first level for the last time we obtain:
\begin{equation*}
\left(p_2,p_1\right)^-\rightarrow\left(t\cdot \psi_{p_2}+p_1,p_2\right)^+,
\end{equation*}
where $\deg(t\cdot \psi_{p_2}+p_1)=\deg p$ (since $\deg p_1<\deg p$), and $\deg p_2=\deg p$. However, in this case $(p+t\cdot \psi_{p_2}+p_1, q+p_2)^+$ will also represent the identity element in $\G$ with $\deg(p+t\cdot \psi_{p_2}+p_1)<\deg p$ and $\deg(q+p_2)<\deg p$, contradicting to the minimality assumption on the degree of $p$ and $q$, unless $p+t\cdot \psi_{p_2}+p_1=0$ and $q+p_2=0$. Therefore, we must have
\[p_2=\psi_{p_1}+\psi_p+q=q\]
or
\[\psi_{p_1}=\psi_{p},\]
which is impossible since $\deg\psi_{p_1}=\deg{p_1}-1<\deg p-1=\deg\psi_p$. Contradiction.
\end{proof}

\begin{lemma}
\label{lem:xyt}
The group $\langle x,y,t\rangle$ is isomorphic to the rank 2 lamplighter group $\Z_2^2\wr\Z=\langle x,y\rangle\wr\langle t\rangle$.
\end{lemma}
\begin{proof}
It is enough to show that the elements $x^{t^n},y^{t^m}$, $n,m\in\Z$ generate $(\Z_2^2)^\infty$. Any possible relation must be of the form
\[1=x^{t^{n_1}}x^{t^{n_2}}\cdots x^{t^{n_r}}\cdot y^{t^{m_1}}y^{t^{m_2}}\cdots y^{t^{m_s}}=x^{t^{n_1}+t^{n_2}+\cdots+t^{n_r}}y^{t^{m_1}+t^{m_2}+\cdots+t^{m_s}},\]
for $n_i,m_j\in\Z$. Conjugating this relation by sufficiently large power of $t$ we obtain $x^{p(t)}y^{q(t)}=1$ for some $p,q\in\Z_2[t]$, which contradicts Lemma~\ref{lem:xpyq}.
\end{proof}

\begin{theorem}
\label{thm:structure_of_G}
The group $\G$ is isomorphic to the index 2 extension of the rank 2 lamplighter group:
\[G\cong \bigl(\Z_2^2\wr\Z\bigr)\rtimes\Z_2=\bigl(\langle x,y\rangle\wr\langle t\rangle\bigr)\rtimes\langle a\rangle,\]
where the action of $a$ on $x,y,t$ is defined as follows: $x^a=x$, $y^a=y^{t^{-1}}$, $t^a=t^{-1}$.
\end{theorem}

\begin{proof}
Follows immediately from Lemma~\ref{lem:xyt} and the fact that $a$ is an involution not belonging to $\langle x,y,t\rangle$.
\end{proof}

\section{State-closed actions of the lamplighter group on the binary tree}
\label{sec:state-closed_binary}

In the Introduction we observed that in all known faithful state-closed representations of lamplighter type groups on the rooted trees the base group is represented by spherically homogeneous automorphisms. In this section we show that modulo conjugation, this is always the case for the lamplighter group $\Z_2\wr\Z$ acting faithfully as an automaton (or state-closed) group on the binary tree. We start by recalling the technique of virtual endomorphisms used, in particular, to study state-closed representations of abelian, nilpotent, and lamplighter type groups in~\cite{nekrash_s:12endomorph},~\cite{berlatto_s:virtual_endom_nilpotent07}, and~\cite{sidki_d:lamplighter}, respectively.

State-closed representations of a group $G$ on $d$-ary tree can be obtained from \emph{similarity pairs} $(H,f)$, where $H$ is a subgroup of index $d$ in $G$ and $f\colon H\to G$ is a homomorphism, called~\emph{virtual endomorphism} of $G$. Each similarity pair $(H,f)$ defines a representation $\phi\colon G\to\Aut(T_d)$ constructed as follows~\cite{nekrash_s:12endomorph}. As before, we will identify $T_d$ with the set $X^*$ of all finite words over the alphabet $X=\{0,1,\ldots,d-1\}$. Let $T=\{e, t_2,\ldots,t_{d-1}\}$ be a right transversal of $H$ in $G$ and $\nu\colon G\to\Sym(T)$ be the permutational
representation of $G$ on $T$. Then for each $g\in G$ we define $\phi(g)$ recursively via the wreath recursion as
\begin{equation}
\label{eqn:virtual_end_rep}
\phi(g)=\bigl(\phi\left(f(h_i)\right)\mid 0\leq i\leq d-1\bigr)\nu(g),
\end{equation}
where $h_i$ are the Schreier elements of $H$ defined by
\[h_i = (t_ig) (t_j)^{-1},\quad t_j = \bigl(\nu(g)\bigr)(t_i).\]

The image $\phi(G)$ is a state-closed subgroup of $\Aut(T_d)$, and the kernel of $\phi$, called the \emph{$f$-core} of $H$, is the largest subgroup $K$ of $H$
which is normal in $G$ and $f$-invariant (in the sense $f(K)<K$); when the kernel is trivial, $f$ and the similarity pair $(H, f)$ are said to be \emph{simple}. In this (and only this) case the representation $\phi$ is faithful.

Let now $G=\Z_2\wr\Z=\langle a\rangle\wr\langle x\rangle$ be the lamplighter group, where $a$ and $x$ denote the standard generators of order 2 and infinity, correspondingly. Let $A$ be the normal closure of $a$ in $G$. Then, since the conjugates of $a$ by powers of $x$ commute, as in~\eqref{eqn:notation_exp}, we can write each element of $A$ as
\begin{equation*}
%\label{eqn:notation_exp}
a^{x^{i_1}+x^{i_2}+\cdots+x^{i_n}}:=a^{x^{i_1}}a^{x^{i_2}}\cdots a^{x^{i_n}}.
\end{equation*}
In particular, $a^{p(x)}\in A$ is defined for each Laurent polynomial $p(x)\in\Z_2[x,x^{-1}]$.

\begin{theorem}
\label{thm:state-closed_lamplighter}
Each state-closed faithful representations of the lamplighter group $\Z_2\wr\Z$ on the binary tree is conjugate to the one with the base group consisting of spherically homogeneous automorphisms.
\end{theorem}

\begin{proof}
The commutator subgroup $[G,G]$ of $G$ is $A_0=\{a^{(1+x)r(x)}\mid r(x)\in\Z_2[x,x^{-1}]\}$. To construct a representation of $G$ on the binary tree we start by picking an index 2 subgroup $H$ in $G$. The only such subgroups in $G$ are $A_0\langle x\rangle$ and $A_0\langle ax\rangle$. We will provide an argument below only for $H=A_0\langle x\rangle$. The case $H=A_0\langle ax\rangle$ is treated similarly.

According to~\cite{sidki_d:lamplighter} a simple virtual endomorphism $f\colon G\to H$ must map $A_0$ to $A$ by
\[f\colon a^{(1+x)r(x)}\mapsto a^{u(x)r(x)},\]
where $1+x$ is not a factor of $u(x)$. To write down the representation $\phi$ of $G$ on the tree $X^*$ for $X=\{0,1\}$ we use decomposition~\eqref{eqn:virtual_end_rep}. First we choose the transversal $T=\{e,a\}$ of $H$ in $G$; any other choice will produce a conjugate representation. Then a straightforward calculation yields:
\[\phi(a)=\bigl(\phi(f((e\cdot a)a^{-1})), \phi(f((a\cdot a)e^{-1}))\bigr)\sigma=(1,1)\sigma,\]
where $\sigma$ is the nontrivial element of $\Sym(X)$. Further,
\begin{multline*}
\phi(x)=\bigl(\phi(f((e\cdot x)e^{-1})), \phi(f((a\cdot x)a^{-1}))\bigr)=\bigl(\phi(f(x)), \phi(f(x^a))\bigr)=\bigl(\phi(f(x)), \phi(f(xa^{1+x}))\bigr)\\
=\bigl(\phi(f(x)), \phi(f(x))\phi(f(a^{1+x}))\bigr)=\bigl(\phi(f(x)), \phi(f(x))\phi(a^{u(x)})\bigr)
=\phi(f(x))^{(1)}\cdot\bigl(1,\phi(a^{u(x)})\bigr).
\end{multline*}
Therefore,
\[
\phi(a^x)=\bigl(\phi(a^{u(x)}),\phi(a^{u(x)})\bigr)\sigma
\]
and using the definition of $f$ for each $r(x)\in\Z_2[x,x^{-1}]$ we obtain that there exists $l(x)\in\Z_2[x,x^{-1}]$ such that
\[
\phi(a^{(1+x)r(x)})=\bigl(\phi(a^{l(x)}),\phi(a^{l(x)})\bigr).
\]

This proves that each element of the base group $A$ maps under the faithful representation $\phi$ to a spherically homogeneous automorphism of $X^*$.
\end{proof}

\section{Open questions}
\label{sec:open}

We conclude the paper with several questions regarding the class of affine automorphisms.

The most natural question is about the general structure of groups generated by automata defining affine automorphisms. The studied examples suggest the following question:
\begin{question}
Is it always the case that a group generated by an automaton defining an affine automorphism is either finite, or is a finite extension of the lamplighter type group of the form $\Z_d^k\wr\Z^l$?
\end{question}

Lamplighter type groups are often defined by reversible and bireversible automata. Our main example in Section~\ref{sec:example} is defined by bireversible automaton, and this is also the case for the 3-state 3-letter automaton generating $\Z_3\wr\Z$ studied by Bondarenko, D'Angeli and Rodaro in~\cite{bondarenko_dr:lamplighter}. The inverse of the standard representation of the lamplighter is defined by a reversible automaton. Since lamplighter groups are closely related to affine automorphisms, it is natural to ask the following question.
\begin{question}
Under what conditions is the automaton defining an affine automorphism reversible? bireversible?
\end{question}

From the algorithmic viewpoint, it is not clear how to check whether a given automaton defines an affine automorphism. In all examples that we dealt with, the proof relied on Theorem~\ref{thm:normalizer} and on a particular structure of a group. This would be useful to have a more uniform and efficient procedure.
\begin{question}
Is there an algorithm deciding whether a given automorphism of the tree defined by a finite initial automaton is affine?
\end{question}

%\bibliographystyle{alpha}
%\bibliography{../../../mylib}

\def\cprime{$'$} \def\cprime{$'$} \def\cprime{$'$} \def\cprime{$'$}
  \def\cprime{$'$} \def\cprime{$'$} \def\cprime{$'$} \def\cprime{$'$}
  \def\cprime{$'$} \def\cprime{$'$} \def\cprime{$'$} \def\cprime{$'$}
  \def\cprime{$'$} \def\cprime{$'$}

\end{document}